\documentclass[12pt]{article}
\usepackage{amsmath,latexsym}
\usepackage{amsthm}
\usepackage{enumerate}
\usepackage{epsfig}

\newcommand\NN{\mathrm{I\!N}}

\newcommand\ZZ{\mathbf{Z}}
\newcommand\eps{{\varepsilon}}

\renewcommand\mod{\,\operatorname{mod}}

\newcommand{\address}{Address: Department of Mathematics, University of North Texas, 1155 Union Circle \#311430, Denton, TX 76203-5017, USA; E-mail: allaart@unt.edu}

\newtheorem{theorem}{Theorem}

\newtheorem{corollary}[theorem]{Corollary}

\newtheorem{proposition}[theorem]{Proposition}

\title{An inequality for sums of binary digits, with application to Takagi functions} 

\author{Pieter C. Allaart \\
University of North Texas \footnote{\address}}
\date{\today}

\begin{document}

\maketitle

\begin{abstract}
Let $\phi(x)=2\inf\{|x-n|:n\in\ZZ\}$, and define for $\alpha>0$ the function
\begin{equation*}
f_\alpha(x)=\sum_{j=0}^\infty \frac{1}{2^{\alpha j}}\phi(2^j x).
\end{equation*}
Tabor and Tabor [{\em J. Math. Anal. Appl.} {\bf 356} (2009), 729--737] recently proved the inequality
\begin{equation*}
f_\alpha\left(\frac{x+y}{2}\right)\leq \frac{f_\alpha(x)+f_\alpha(y)}{2}+|x-y|^\alpha,
\end{equation*}
for $\alpha\in[1,2]$. By developing an explicit expression for $f_\alpha$ at dyadic rational points, it is shown in this paper that the above inequality can be reduced to a simple inequality for weighted sums of binary digits. That inequality, which seems of independent interest, is used to give an alternative proof of the result of Tabor and Tabor, which captures the essential structure of $f_\alpha$.



\end{abstract}

\section{Introduction}

Let $\phi(x)=2\inf\{|x-n|:n\in\ZZ\}$ be the so-called ``tent-map," and define for $\alpha>0$ the function
\begin{equation}
f_\alpha(x)=\sum_{j=0}^\infty \frac{1}{2^{\alpha j}}\phi(2^j x).
\label{eq:definition-of-f}
\end{equation}
Observe that $f_1$ is two times Takagi's continuous nowhere differentiable function; see \cite{Takagi}. For $0<\alpha<1$, the graph of $f_\alpha$ is a fractal whose Hausdorff dimension was calculated by Ledrappier \cite{Ledrappier}. For $\alpha>1$, the function $f_\alpha$ is Lipschitz and hence differentiable almost everywhere. The special case $\alpha=2$ gives the only smooth function in this family, as $f_2(x)=4x(1-x)$. 

This paper concerns the following inequality, proved recently by Tabor and Tabor \cite{Tabor}.

\begin{theorem} \label{thm:Tabor}
{\rm (Tabor and Tabor \cite[Corollary 2.1]{Tabor})}.
For every $1\leq\alpha\leq 2$ and for all $x,y\in[0,1]$,
\begin{equation}
f_\alpha\left(\frac{x+y}{2}\right)\leq \frac{f_\alpha(x)+f_\alpha(y)}{2}+|x-y|^\alpha.
\label{eq:approximate-convexity}
\end{equation}
\end{theorem}

This inequality plays an important role in the study of approximate convexity of continuous functions, where $f_\alpha$ occurs naturally in a best possible upper bound; see \cite{Tabor}. For the case $\alpha=1$, the inequality had previously been proved by Boros \cite{Boros}. Note that for $\alpha=2$, \eqref{eq:approximate-convexity} holds with equality for all $x$ and $y$ in $[0,1]$. Both Boros' proof and Tabor and Tabor's proof of \eqref{eq:approximate-convexity}, while cleverly devised, provide little insight into the essential structure of the function $f_\alpha$. The aim of this note is to show how \eqref{eq:approximate-convexity} can be reduced to a simple inequality concerning weighted sums of binary digits, thereby providing a simpler proof for the inequality \eqref{eq:approximate-convexity} that emphasizes the basic structure of $f_\alpha$.

We need the following notation. For a nonnegative integer $n$ and a real number $p$, write $n$ in binary as $n=\sum_{j=0}^\infty 2^j\eps_j$ with $\eps_j\in\{0,1\}$, and define
\begin{equation}
s_p(n)=\sum_{j=0}^\infty 2^{pj}\eps_j.
\label{eq:binary-sum}
\end{equation}
Let
\begin{equation*}
S_p(n)=\sum_{m=0}^{n-1}s_p(m).
\end{equation*}
It turns out that \eqref{eq:approximate-convexity} is equivalent to the simple inequality
\begin{equation}
S_p(m+l)+S_p(m-l)-2S_p(m)\leq l^{p+1}, 
\label{eq:main-inequality}
\end{equation}
for $0\leq p\leq 1$ and $0\leq l\leq m$. This inequality, which seems to be of independent interest, is proved in Section \ref{sec:digitsum}; there we also specify the cases when equality holds in \eqref{eq:main-inequality}. Note that when $p=0$, $S_p(n)$ is the number of $1$'s needed to 
express the numbers $0,\dots,n-1$ in binary. Since we can write \eqref{eq:main-inequality} in this case as
\begin{equation*}
S_0(m+l)-S_0(m)-[S_0(m)-S_0(m-l)]\leq l,
\end{equation*}
it follows that for any list of $2l$ consecutive positive integers $m-l,\dots,m+l-1$, the number of binary $1$'s needed to write the second half of the list (the numbers $m,\dots,m+l-1$) is at most $l$ more than the number of $1$'s required to write the first half (the numbers $m-l,\dots,m-1$). The function $S_0$ has been well studied in the literature; see, for instance, Trollope \cite{Trollope} for a precise expression and asymptotics. When $p=1$, $S_p(n)$ is simply the sum of the first $n-1$ positive integers, from which it follows readily that \eqref{eq:main-inequality} holds with equality for all $l$ and $m$. It seems that for $0<p<1$ the inequality may be new. In fact, even for the case $p=0$ the author has not been able to find a reference.

The key to showing that \eqref{eq:approximate-convexity} reduces to \eqref{eq:main-inequality} is the following formula for the values of $f_\alpha$ at dyadic rational points.

\begin{proposition} \label{prop:expression}
For $n=0,1,\dots$ and $m=0,1,\dots,2^n$,
\begin{equation}
f_\alpha\left(\frac{m}{2^n}\right)=\sum_{k=0}^{m-1}\sum_{i=0}^{n-1}\frac{{(-1)}^{\eps_i(k)}}{2^{(n-i-1)\alpha+i}},
\label{eq:expression}
\end{equation}
where $\eps_i(k)\in\{0,1\}$ is determined by $\sum_{i=0}^{n-1}2^i\eps_i(k)=k$.
\end{proposition}

For $\alpha=1$, this formula simplifies to a well-known expression for the Takagi function; see, for instance, Kr\"uppel \cite[eq. (2.4)]{Kruppel}. The formula in its general form above does not seem to have been published before, and could be useful for studying a variety of other properties of the functions $f_\alpha$, including their level sets and finer differentiability structure. For the Takagi function (i.e. $f_1$), the level sets were considered for instance by Maddock \cite{Maddock}, and a description of the set of points $x$ with $f_1'(x)=\pm\,\infty$ was given by Allaart and Kawamura \cite{AK}. In these papers, explicit expressions such as \eqref{eq:expression} above played an important role.

Proposition \ref{prop:expression} is proved in Section \ref{sec:Takagi}. It is then used, together with \eqref{eq:main-inequality}, to give a short proof of Theorem \ref{thm:Tabor}.

\section{A digital sum inequality} \label{sec:digitsum}

This section gives a proof of the inequality \eqref{eq:main-inequality}, 
and specifies in which cases equality holds.

\begin{theorem} \label{thm:main}
Let $0\leq p\leq 1$. Then \eqref{eq:main-inequality} holds for all $m\geq  0$ and $0\leq l\leq m$. Moreover, if $l\geq 1$ and $k$ is the integer such that $2^{k-1}<l\leq 2^k$, then equality holds in \eqref{eq:main-inequality} if and only if one of the following holds:\vspace{-0.2em}
\begin{enumerate}[(i)]\setlength{\itemsep}{-0.2em}
\item $p=1$; or
\item $0<p<1$, $l=2^k$, and $m\equiv l \mod 2^{k+1}$ ; or
\item $p=0$, and either $m\equiv l \mod 2^{k+1}$ or $m\equiv -l \mod 2^{k+1}$.
\end{enumerate}
\end{theorem}



\begin{proof}
{\bf Part 1: Inequality.} 
Fix $p\in[0,1]$. For brevity, write
\begin{equation*}
\Delta(m,l):=S_p(m+l)+S_p(m-l)-2S_p(m).
\end{equation*}
Note that the inequality is trivially satisfied when $l=0$.
The proof proceeds by induction on $l$. First, let $l=1$, and note that in this case,
\begin{equation*}
\Delta(m,1)=s_p(m)-s_p(m-1).
\end{equation*}
Consider two cases regarding the parity of $m$. If $m$ is odd, then $\eps_0(m-1)=0$ and $\eps_0(m)=1$, while $\eps_j(m-1)=\eps_j(m)$ for all $j\geq 1$. Hence, $\Delta(m,1)=1$.

Assume then that $m$ is even. In this case, there is $j_0\geq 1$ such that: 
\begin{gather*}
\eps_j(m-1)=1 \quad\mbox{and} \quad \eps_j(m)=0\quad \mbox{for $0\leq j<j_0$, }\\ 
\eps_{j_0}(m-1)=0 \quad\mbox{and} \quad \eps_{j_0}(m)=1,\qquad \mbox{and}\\
\eps_j(m-1)=\eps_j(m)\quad \mbox{for all $j>j_0$}.
\end{gather*}
Thus,
\begin{equation}
\Delta(m,1)=2^{j_0 p}-\sum_{j=0}^{j_0-1}2^{jp}.
\label{eq:delta_m1}
\end{equation}
If $p=0$, it follows immediately that $\Delta(m,1)<1$. If $0<p\leq 1$, we may put $\lambda=2^p$ and obtain that
\begin{equation}
\Delta(m,1)-1=\lambda^{j_0}-1-\frac{\lambda^{j_0}-1}{\lambda-1}=(\lambda^{j_0}-1)\frac{\lambda-2}{\lambda-1}\leq 0.
\label{eq:basis-case}
\end{equation}
Thus, \eqref{eq:main-inequality} holds for $l=1$ and all $m\geq 1$. In fact, if $l=1$ and $p<1$, it is clear from \eqref{eq:delta_m1} and \eqref{eq:basis-case} that equality obtains in \eqref{eq:main-inequality} if and only if $m$ is odd.

Next, let $n\geq 2$, and assume that $\Delta(m,l)\leq l^{p+1}$ for all $l<n$ and all $m$. For ease of notation, put
\begin{equation*}
\Sigma(t,u):=\sum_{r=t}^{u-1}s_p(r)=S_p(u)-S_p(t), \qquad t,u\in\NN,\ \ t<u.
\end{equation*}
Let $k$ be the integer such that $2^{k-1}<n\leq 2^k$. The idea is to write
\begin{align*}
\Delta(m,n)&=\Sigma(m-n+2^k,m+n)-\Sigma(m-n,m+n-2^k)\\
&\qquad+\Sigma(m,m-n+2^k)-\Sigma(m+n-2^k,m).
\end{align*}
(See Figure \ref{fig:illustration}, which also illustrates the next few steps of the induction argument.)

\begin{figure}
\begin{align*}
\left.
\begin{array}{r}
m-n\\
\ 
\end{array} \quad
\begin{array}{ccc}
1 & 0 & 1\\
1 & 0 & 1
\end{array}
\begin{array}{|ccc|}
\hline
0 & 1 & 0\\
0 & 1 & 1\\
\hline
\end{array} \quad
\right\}\,\quad &\Sigma(m-n,m+n-2^k)\\
\left.
\begin{array}{r}
m+n-2^k\\
\ \\
\ 
\end{array} \quad
\begin{array}{cccccc}
 1 & 0 & 1 & 1 & 0 & 0\\
 1 & 0 & 1 & 1 & 0 & 1\\
 1 & 0 & 1 & 1 & 1 & 0\\
\hline
\end{array}
\quad
\right\} \quad &\Sigma(m+n-2^k,m)\\
\left.
\begin{array}{r}
m\\
\ \\
\ 
\end{array} \quad
\begin{array}{ccc}
 1 & 0 & 1\\
 1 & 1 & 0\\
 1 & 1 & 0
\end{array}
\begin{array}{ccc}
1 & 1 & 1\\
0 & 0 & 0\\
0 & 0 & 1
\end{array} \quad
\right\} \quad &\Sigma(m,m-n+2^k)\\
\left.
\begin{array}{r}
m-n+2^k\\
\ 
\end{array} \quad
\begin{array}{ccc}
 1 & 1 & 0\\
 1 & 1 & 0
\end{array}
\begin{array}{|ccc|}
\hline
0 & 1 & 0\\
0 & 1 & 1\\
\hline
\end{array} \quad
\right\}\,\quad &\Sigma(m-n+2^k,m+n)
\end{align*}
\caption{The induction step in the proof of  \eqref{eq:main-inequality} illustrated for $m=47$ and $n=5$.}
\label{fig:illustration}
\end{figure}

Since the list $m,\dots,m-n+2^k-1$ has $2^k-n$ elements and $2^k-n<n$, the induction hypothesis implies that
\begin{equation}
\Sigma(m,m-n+2^k)-\Sigma(m+n-2^k,m)=\Delta(m,2^k-n)\leq {(2^k-n)}^{p+1}.
\label{eq:middle-group}
\end{equation}
On the other hand, for $r=0,1,\dots,2n-2^k-1$, the numbers $m-n+r$ and $m-n+2^k+r$ have their $k$ least significant binary digits in common (see the boxes in Figure \ref{fig:illustration}), and so
\begin{equation}
s_p(m-n+2^k+r)-s_p(m-n+r)=2^{kp}\{s_p(t+1)-s_p(t)\},
\label{eq:one-by-one}
\end{equation}
where $t$ is the greatest integer in $(m-n+r)/2^k$; this follows because the terms for $j=0,\dots,k-1$ in the definition \eqref{eq:binary-sum} cancel each other in the left hand side above. (For example, in Figure \ref{fig:illustration} we have $m=47$ and $n=5$, so $k=3$, and $t=5$ for both $r=0$ and $r=1$.)
By \eqref{eq:one-by-one} and the fact that \eqref{eq:main-inequality} holds for the case $l=1$,
\begin{equation}
\Sigma(m-n+2^k,m+n)-\Sigma(m-n,m+n-2^k)\leq (2n-2^k)\cdot 2^{kp},
\label{eq:outside-groups}
\end{equation}
with strict inequality when $p<1$ and $t$ is odd.
Combining \eqref{eq:middle-group} and \eqref{eq:outside-groups}, we obtain
\begin{align*}
\Delta(m,n)&\leq (2n-2^k)\cdot 2^{kp}+{(2^k-n)}^{p+1}\\
&=2^{k(p+1)}\left[2\left(\frac{n}{2^k}\right)-1+\left(1-\frac{n}{2^k}\right)^{p+1}\right].
\end{align*}
Put $x=n/2^k$. Then $1/2<x\leq 1$, and it will follow that $\Delta(m,n)\leq n^{p+1}$ provided that 
\begin{equation}
2x-1+(1-x)^{p+1}\leq x^{p+1}.
\label{eq:elementary-inequality}
\end{equation}
But this last inequality follows since the function 
\begin{equation}
g_p(x):=2x-1+(1-x)^{p+1}-x^{p+1}
\label{eq:g}
\end{equation}
is convex on $[1/2,1]$ for $p\in[0,1]$, with $g_p(1/2)=g_p(1)=0$. This concludes the inductive proof of the inequality \eqref{eq:main-inequality}. 

(It is worth noting that \eqref{eq:elementary-inequality} was used also by Tabor and Tabor \cite{Tabor} in their proof of \eqref{eq:approximate-convexity}.)

\bigskip
{\bf Part 2: Equality.}
We now turn to the question of equality. It was noted in the introduction that if $p=1$, then $s_p(n)=n$, and so $\Delta(m,l)=l^2$ for all $l$ and $m$. Suppose $0<p<1$. If $l=2^k$ and $m\equiv l \mod 2^{k+1}$, then 
\begin{equation*}
\Delta(m,l)=l\cdot 2^{kp}=l^{p+1}.
\end{equation*}
On the other hand, if $l=2^k$ but $m\not\equiv l \mod 2^{k+1}$, then strict inequality obtains in \eqref{eq:outside-groups} in the induction step, as the greatest integer in $(m-l+r)/2^k$ is odd for at least one $r$. Finally, if $l<2^k$, then with $x=l/2^k$ we have strict inequality in \eqref{eq:elementary-inequality}, since the function $g_p$ defined in \eqref{eq:g} is strictly convex on $[1/2,1]$ when $0<p<1$.

The case $p=0$ is the most involved. We will show inductively that $\Delta(m,l)=l$ if and only if $m\equiv \pm\, l\mod 2^{k+1}$. Note that this equivalence holds for the case $l=1$ by the remark following \eqref{eq:basis-case}. 

Let $n\geq 2$, and assume that whenever $l<n$ and $j$ is the integer such that $2^{j-1}<l\leq 2^j$, the equivalence
\begin{equation*}
m\equiv \pm\, l\mod 2^{j+1}\quad\Longleftrightarrow\quad\Delta(m,l)=l
\end{equation*}
holds. Let $k$ be the integer such that $2^{k-1}<n\leq 2^k$, and put 
\begin{equation*}
l:=2^k-n.
\end{equation*}
Observe that $l<2^{k-1}<n$.

Suppose $m\equiv \pm\, n \mod 2^{k+1}$. Then either the binary representation of $m-n$ ends in $k$ zeros, or that of $m+n-1$ ends in $k$ ones. In both cases, 
\begin{equation}
\eps_k(m-n+2^k+r)=1\quad\mbox{and}\quad \eps_k(m-n+r)=0,\quad\mbox{for $0\leq r<2n-2^k$},
\label{eq:these-are-different}
\end{equation}
so 
\begin{equation}
\Sigma(m-n+2^k,m+n)-\Sigma(m-n,m+n-2^k)=2n-2^k=n-l.
\label{eq:outside}
\end{equation}
If $l=0$ the two middle groups vanish, so $\Delta(m,n)=n-l=n$.
Assume then that $l>0$. Let $j$ be the integer such that $2^{j-1}<l\leq 2^j$. Since $l<2^{k-1}$, we have $j<k$. If $m\equiv n\mod 2^{k+1}$, then $m+l\equiv 2^k\mod 2^{k+1}$ and hence $m+l\equiv 0\mod 2^{j+1}$. Similarly, if $m\equiv -n\mod 2^{k+1}$, then $m-l\equiv 0\mod 2^{j+1}$. Thus, by the induction hypothesis, 
\begin{equation*}
\Sigma(m+n-2^k,m)-\Sigma(m,m-n+2^k)=\Delta(m,l)=l. 
\end{equation*}
Combining this with \eqref{eq:outside} yields $\Delta(m,n)=n$.

Conversely, suppose $\Delta(m,n)=n$. Then equality must hold in both \eqref{eq:middle-group} and \eqref{eq:outside-groups}, so in particular,
\begin{equation*}
s_0(m-n+2^k+r)-s_0(m-n+r)=1, \qquad\mbox{for $0\leq r<2n-2^k$}.
\end{equation*}
This implies \eqref{eq:these-are-different}. If $n=2^k$, it follows immediately that $m\equiv n\mod 2^{k+1}$. Otherwise, $l>0$, and we let $j$ be the integer such that $2^{j-1}<l\leq 2^j$. Since
\begin{equation*}
\Delta(m,l)=\Sigma(m,m+l)-\Sigma(m-l,m)=l,
\end{equation*}
the induction hypothesis implies that $m\equiv \pm\,l\mod 2^{j+1}$. If $m\equiv l\mod 2^{j+1}$, then the binary expansion of $m-l$ ends in $j+1$ zeros. The set $A:=\{m-l,\dots,m+l-1\}$ contains $2l\leq 2^{j+1}$ numbers, so $\eps_i(\cdot)$ is constant on $A$ for each $i>j$. In particular, $\eps_k(\cdot)$ is constant on $A$, since $k>j$. The same conclusion results if $m\equiv -l\mod 2^{j+1}$, as then the binary expansion of $m+l-1$ ends in $j+1$ ones.

Suppose the common value of $\eps_k(\cdot)$ on $A$ is $0$. Since $A$ contains the numbers $m-n+r$ where $r=2n-2^k,\dots,2^k-1$, we obtain by \eqref{eq:these-are-different} that $\eps_k(m-n+r)=0$ for $r=0,\dots,2^k-1$, and so $m-n\equiv 0\mod 2^{k+1}$. On the other hand, suppose the common value of $\eps_k(\cdot)$ on $A$ is $1$. Then by \eqref{eq:these-are-different}, $\eps_k(m-n+r)=1$ for $r=2n-2^k,\dots,2n-1$, or equivalently (putting $r'=2n-r$), $\eps_k(m+n-r')=1$ for $r'=1,\dots,2^k$. But this implies $m+n\equiv 0\mod 2^{k+1}$. In either case, $m\equiv \pm\,n \mod 2^{k+1}$, as desired. Thus, the proof is complete.
\end{proof}

\section{Application to Takagi functions} \label{sec:Takagi}

This section gives a proof of Proposition \ref{prop:expression}, and shows how the expression given in the proposition can be used, in conjunction with the inequality \eqref{eq:main-inequality}, to give a more straightforward proof of the theorem of Tabor and Tabor.

\begin{proof}[Proof of Proposition \ref{prop:expression}]
Since $\phi$ vanishes at integer points, the definition \eqref{eq:definition-of-f} of $f_\alpha$ gives
\begin{gather}
f_\alpha\left(\frac{j}{2^n}\right)=\sum_{m=0}^{n-1}\frac{1}{2^{\alpha m}}\phi\left(\frac{j}{2^{n-m}}\right), \label{eq:at-left}\\
f_\alpha\left(\frac{j+1}{2^{n}}\right)=\sum_{m=0}^{n-1}\frac{1}{2^{\alpha m}}\phi\left(\frac{j+1}{2^{n-m}}\right), \label{eq:at-right}
\end{gather}
and
\begin{equation*}
f_\alpha\left(\frac{2j+1}{2^{n+1}}\right)=\sum_{m=0}^n \frac{1}{2^{\alpha m}}\phi\left(\frac{2j+1}{2^{n+1-m}}\right).
\end{equation*}
Since $\phi$ is linear on each interval $[j/2,(j+1)/2]$ with $j\in\ZZ$,
\begin{equation*}
\phi\left(\frac{2j+1}{2^{n+1-m}}\right)=\frac12\left\{\phi\left(\frac{j}{2^{n-m}}\right)+\phi\left(\frac{j+1}{2^{n-m}}\right)\right\}.
\end{equation*}
Noting also that $\phi\big((2j+1)/2\big)=1$ for all $j\in\ZZ$, we thus obtain
\begin{equation}
f_\alpha\left(\frac{2j+1}{2^{n+1}}\right)=\frac12\left\{f_\alpha\left(\frac{j}{2^n}\right)+f_\alpha\left(\frac{j+1}{2^n}\right)\right\}+\frac{1}{2^{\alpha n}},
\label{eq:midpoint-recursion}
\end{equation}
for $n=0,1,\dots$, and $j=0,1,\dots,2^n-1$.
From this, it follows that
\begin{equation}
f_\alpha\left(\frac{k+1}{2^{n+1}}\right)-f_\alpha\left(\frac{k}{2^{n+1}} \right)=\frac12\left\{f_\alpha\left(\frac{j+1}{2^n}\right)-f_\alpha\left(\frac{j}{2^n}\right)\right\}+\frac{(-1)^k}{2^{\alpha n}},
\label{eq:difference-expression}
\end{equation}
where $j=[k/2]$ is the greatest integer in $k/2$. This last equation follows easily from \eqref{eq:midpoint-recursion} by considering separately the cases $k=2j$ and $k=2j+1$.
A straightforward induction argument using \eqref{eq:difference-expression} yields
\begin{equation*}
f_\alpha\left(\frac{k+1}{2^{n+1}}\right)-f_\alpha\left(\frac{k}{2^{n+1}}\right)=\sum_{i=0}^n \frac{{(-1)}^{\eps_i}}{2^{(n-i)\alpha+i}},
\end{equation*}
where $k=\sum_{i=0}^n 2^i\eps_i$. Replacing $n$ with $n-1$ and summing over $k=0,\dots,m-1$ gives \eqref{eq:expression}, as $f_\alpha(0)=0$.
\end{proof}

\begin{proof}[Proof of Theorem \ref{thm:Tabor}]
Since $f_\alpha$ is continuous, it suffices to prove \eqref{eq:approximate-convexity} for dyadic rational points $x$ and $y$. Thus, we may assume that there exist nonnegative integers $n,m$ and $l$ such that $x=(m-l)/2^n$ and $y=(m+l)/2^n$. It is to be shown that
\begin{equation*}
\Delta_2^{(n)}(m,l):=f_\alpha\left(\frac{m}{2^n}\right)-\frac12\left\{f_\alpha\left(\frac{m-l}{2^n}\right)+f_\alpha\left(\frac{m+l}{2^n}\right)\right\}\leq {\left(\frac{2l}{2^n}\right)}^\alpha.
\end{equation*}
Proposition \ref{prop:expression} gives
\begin{align*}
\Delta_2^{(n)}(m,l)&=\sum_{i=0}^{n-1}\frac{1}{2^{(n-i-1)\alpha+i}}\left(\sum_{k=0}^{m-1}{(-1)}^{\eps_i(k)}-\frac12\sum_{k=0}^{m-l-1}{(-1)}^{\eps_i(k)}
-\frac12\sum_{k=0}^{m+l-1}{(-1)}^{\eps_i(k)}\right)\\
&=\frac{1}{2^{(n-1)\alpha}}\sum_{i=0}^{n-1}2^{(\alpha-1)i-1}\left(\sum_{k=m-l}^{m-1}{(-1)}^{\eps_i(k)}-\sum_{k=m}^{m+l-1}{(-1)}^{\eps_i(k)}\right).
\end{align*}
Since $(-1)^\eps=1-2\eps$ for $\eps\in\{0,1\}$, we can write
\begin{align*}
2^{(n-1)\alpha}\Delta_2^{(n)}(m,l)&=\sum_{i=0}^{n-1}2^{(\alpha-1)i}\sum_{r=1}^l\left\{\eps_i(m+r-1)-\eps_i(m-r)\right\}\\
&=\sum_{r=1}^l\left\{s_{\alpha-1}(m+r-1)-s_{\alpha-1}(m-r)\right\}\\
&=S_{\alpha-1}(m+l)+S_{\alpha-1}(m-l)-2S_{\alpha-1}(m).
\end{align*}
Thus, by Theorem \ref{thm:main},
\begin{equation*}
\Delta_2^{(n)}(m,l)\leq \frac{l^\alpha}{2^{(n-1)\alpha}}={\left(\frac{2l}{2^n}\right)}^\alpha,
\end{equation*}
as required.
\end{proof}

In fact, it is not difficult to use Theorem \ref{thm:main} to determine for which dyadic points $x$ and $y$ equality holds in \eqref{eq:approximate-convexity}. It was already remarked in the introduction that equality holds for all $x$ and $y$ in $[0,1]$ when $\alpha=2$. 

\begin{corollary}
(i) If $1<\alpha<2$, then equality holds in \eqref{eq:approximate-convexity} for dyadic rationals $x$ and $y$ in $[0,1]$ if and only if there exist integers $j$ and $r$ such that $x=j/2^r$ and $|x-y|=1/2^r$.

(ii) If $\alpha=1$, then equality holds in \eqref{eq:approximate-convexity} for dyadic rationals $x$ and $y$ in $[0,1]$ if and only if there exist integers $j$ and $r$ such that either $x=j/2^r$ or $y=j/2^r$, and $|x-y|\leq 1/2^r$.
\end{corollary}

\begin{proof}
The proof of Theorem \ref{thm:Tabor} shows that for $x=(m-l)/2^n$ and $y=(m+l)/2^n$, equality holds in \eqref{eq:approximate-convexity} if and only if equality holds in \eqref{eq:main-inequality} with $p=\alpha-1$.

(i) Let $1<\alpha<2$. If $x=j/2^r$ and $y=(j+1)/2^r$, then equality holds in \eqref{eq:approximate-convexity} by \eqref{eq:midpoint-recursion}. 
Conversely, if $x$ and $y$ are dyadic rationals in $[0,1]$ with $x<y$  which attain equality in \eqref{eq:approximate-convexity}, then we can write $x=(m-l)/2^n$ and $y=(m+l)/2^n$ for integers $l,m$ and $n$, and it follows from Theorem \ref{thm:main} (with $0<p<1$) that there exist integers $k$ and $j$ such that $l=2^k$, and $m=(2j+1)2^k$. But then, putting $r=n-k-1$, we get $x=j/2^r$ and $y=(j+1)/2^r$.

(ii) Let $\alpha=1$, and suppose first that $x=j/2^r$ and $y$ is dyadic rational with $x<y\leq (j+1)/2^r$. Replacing $r$ with a greater integer if necessary, we may assume that $1/2^{r+1}<y-x\leq 1/2^r$. Write $y-x=l/2^s$, where $l,s\in \ZZ$. Now put $n=s+1$, $k=s-r=n-r-1$, and $m=l+2^{k+1}j$. Then $2^{k-1}<l\leq 2^k$, $m\equiv l \mod 2^{k+1}$, $x=(m-l)/2^n$, and $y=(m+l)/2^n$. So by Theorem \ref{thm:main} (with $p=0$), $x$ and $y$ satisfy \eqref{eq:approximate-convexity} with equality. The case $(j-1)/2^r\leq y<x$ follows similarly.

Conversely, let $x$ and $y$ be dyadic rationals in $[0,1]$ with $x<y$  which attain equality in \eqref{eq:approximate-convexity}. Then we can write $x=(m-l)/2^n$ and $y=(m+l)/2^n$ for integers $l,m$ and $n$. Theorem \ref{thm:main} implies that $m\equiv \pm\,l \mod 2^{k+1}$, so there is $j\in\ZZ$ such that either $m=l+2^{k+1}j$ or $m=-l+2^{k+1}j$, where $k$ is the integer such that $2^{k-1}<l\leq 2^k$. Suppose $m=l+2^{k+1}j$, and put $r=n-k-1$. Then $x=j/2^r$ and $y=x+2l/2^n$, so $y$ is a dyadic rational with $y\leq x+1/2^r$, since $2l/2^n\leq 2^{k+1}/2^n=1/2^r$. The case $m=-l+2^{k+1}j$ is similar, leading to $y=j/2^r$ and $x=y-2l/2^n$, so again $|y-x|\leq 1/2^r$.
\end{proof}

Note that by continuity of $f_1$, it follows from statement (ii) in the Corollary that for $\alpha=1$, equality holds in \eqref{eq:approximate-convexity} whenever $x=j/2^r$ and $y$ is any {\em real} number with $|x-y|\leq 1/2^r$. This is not surprising, considering the self-affine structure of $f_1$. However, it seems difficult to determine whether equality can hold in \eqref{eq:approximate-convexity} when $x$ and $y$ are both nondyadic.

\section*{Acknowledgment}
The author wishes to thank two anonymous referees for a careful reading of the manuscript and for several suggestions which improved the presentation of the paper.


\footnotesize

\begin{thebibliography}{8}

\bibitem{AK}
{\sc P. C. Allaart} and {\sc K. Kawamura}, The improper infinite derivatives of Takagi's nowhere-differentiable function, {\em J. Math. Anal. Appl.} {\bf 372} (2010), no. 2, 656--665.

\bibitem{Boros}
{\sc Z. Boros}, An inequality for the Takagi function. {\em Math. Inequal. Appl.} {\bf 11} (2008), no. 4, 757--765.

\bibitem{Kruppel}
{\sc M.~Kr\"uppel}, On the extrema and the improper derivatives of Takagi's continuous nowhere differentiable function, {\em Rostock. Math. Kolloq.} {\bf 62} (2007), 41--59.

\bibitem{Ledrappier} 
{\sc F. Ledrappier}, On the dimension of some graphs, {\em Contemp. Math.} {\bf 135} (1992), 285--293.

\bibitem{Maddock}
{\sc Z. Maddock}, Level sets of the Takagi function: Hausdorff dimension, {\em Monatsh. Math.} {\bf 160} (2010), no. 2, 167--186.

\bibitem{Tabor}
{\sc J. Tabor} and {\sc J. Tabor}, Takagi functions and approximate midconvexity. {\em J. Math. Anal. Appl.} {\bf 356} (2009), no. 2, 729--737.

\bibitem{Takagi} 
{\sc T.~Takagi}, A simple example of the continuous function without derivative, {\em Phys.-Math. Soc. Japan} {\bf 1} (1903), 176-177. {\em The Collected Papers of Teiji Takagi}, S. Kuroda, Ed., Iwanami (1973), 5--6. 

\bibitem{Trollope}
{\sc J. R. Trollope}, An explicit expression for binary digital sums, {\em Math. Mag.} {\bf 41} (1968), 21-25.


\end{thebibliography}
\end{document}